\documentclass[12pt]{article}
\usepackage{amsmath,amsthm,amsfonts,latexsym,amsopn,verbatim,amscd,amssymb}
\theoremstyle{plain}
\newtheorem{thm}{Theorem}[section]
\newtheorem{lemma}[thm]{Lemma}

\newtheorem{prop}[thm]{Proposition}

\newtheorem{Question}[thm]{Question}
\theoremstyle{definition}

\newcommand{\Z}{\mathbb{Z}}

\newcommand{\bnum}{\begin{enumerate}}
\newcommand{\enum}{\end{enumerate}}

\numberwithin{equation}{section}

\DeclareMathOperator{\Cent}{Cent}
\begin{document}
\title{\textbf{Characterizing some rings of finite order}}
\author{Jutirekha Dutta, Dhiren K. Basnet and Rajat K. Nath\footnote{Corresponding author}}
\date{}
\maketitle
\begin{center}\small{\it 
Department of Mathematical Sciences, Tezpur University,\\ Napaam-784028, Sonitpur, Assam, India.\\
%\thanks{*Corresponding author}
%Department of Mathematical Sciences,   Tezpur University,\\ Napaam-784028, Sonitpur, Assam, India.\\

%Department of Mathematical Sciences,   Tezpur University,\\ Napaam-784028, Sonitpur, Assam, India.\\
Emails:\, jutirekhadutta@yahoo.com,  dbasnet@tezu.ernet.in and rajatkantinath@yahoo.com*}
\end{center}

\medskip

\begin{abstract} 
In this paper, we compute the number of distinct centralizers of some classes of finite rings. We then characterize all finite rings with $n$ distinct centralizers for any positive integer $n \leq 5$. Further we give some connections between the number of distinct centralizers of  a finite ring and its commutativity degree. 
\end{abstract}

\medskip

 

\medskip

\section{Introduction}
Finite abelian groups have been completely characterized up to isomorphism for a long time but finite rings have yet to be characterized. The problem of characterizing finite rings up to isomorphism  has received considerable attention in recent years (see \cite{bBhK14, Cu10, dOP94, Fin93, gC95}) starting from the works of  Eldridge \cite{El68} and   Raghavendran \cite{Ra70}. In this paper we characterize finite rings in terms of their number of distinct centralizers. Given a ring $R$ and an element $r \in R$, the subrings $C(r)   =   \{s \in R \colon rs \; = \; sr\}$ and $Z(R) = \{ s \in R \;:\; rs = sr \text{ for all } r \in R\}$  are known as \textit{centralizer} of $r$ in $R$ and \textit{center} of $R$ respectively. We write $\Cent(R)$ to denote the set of all  centralizers in $R$. Firstly we compute the order of  $\Cent (R)$ for  some classes of finite rings  $R$.  Motivated by the works of Belcastro and Sherman \cite{bG94} and   Ashrafi \cite{ashrafi00}, we define   $n$-centralizer ring for any positive integer $n$. A ring $R$ is said to be    \textit{$n$-centralizer} ring if $|\Cent (R)| = n$, for any positive integer $n$.  We then characterize  $n$-centralizer finite rings for all  $n \leq 5$, adapting similar techniques  that are used by Belcastro and Sherman \cite{bG94} in order to characterize $n$-centralizer finite groups for $n \leq 5$.  

%We would like to mention here that   $n$-centralizer finite groups for $n \leq 5$  have been characterized by   Belcastro and Sherman \cite{bG94}.

Further, we conclude the paper by noting some interesting connections between $d(R)$ and $|\Cent (R)|$. Note that for any finite ring $R$, the ratio $d(R) = \frac{1}{|R|^2}\underset{r \in R}{\sum}|C(r)|$ is the probability that a randomly chosen pair of elements of $R$ commute.  This ratio is known as \textit{commutativity degree} of  finite ring $R$ and it was introduced by MacHale \cite{dmachale} in the year 1975. Some characterizations of finite rings in terms of commutativity degree can be found in \cite{dmachale, BMS, BM}.     

  Throughout the paper $R$ denotes a finite ring. For any subring $S$ of $R$,   $R/S$  denotes the additive quotient group  and $|R : S|$ denotes the index of  the additive subgroup $S$ in the additive group $R$. Note that the isomorphisms considered are the additive group isomorphisms. Also for any two non-empty subsets $A$ and $B$ of a ring $R$, we write $A + B = \{a + b : a \in A, b \in B\}$. We shall use the fact that  for any non-commutative ring $R$, the additive group $\frac{R}{Z(R)}$ is not a cyclic group (see \cite[Lemma 1]{dmachale}).

\section{Some computations of $|\Cent(R)|$}
In this section, we compute $|\Cent(R)|$ for some classes of finite rings. However,  first we prove some   results which are   useful for subsequent results as well as for the next sections.

\begin{prop}\label{prop1}
  $R$ is a commutative ring if and only if  $R$ is a $1$-centralizer ring.
\end{prop} 
\begin{proof}
The proposition follows from the fact that a ring $R$ is  commutative if and only if $C(r) = R$ for each $r \in R$. 
\end{proof}

\begin{prop}\label{prop2}
Let $R, S$ be two rings, then 
\[ 
\Cent(R \times S) = \Cent(R) \times \Cent(S).
\]
\end{prop}
\begin{proof}
It can be easily seen that $C((r, s)) = C(r) \times C(s)$ for any $r \in R$ and $s \in S$. This proves the proposition.
\end{proof}

The following lemmas play an important role in finding lower bound of $| \Cent(R)|$ for any non-commutative ring $R$. 

\begin{lemma}\label{lemma1}
Let $R$ be a ring. Then $Z(R)$ is the intersection of all centralizers in $R$.
\end{lemma}
\begin{proof}
It is clear that $Z(R) \; \subseteq \; \underset{r \in R}{\cap} C(r)$. Now, for any $s \in \underset{r \in R}{\cap} C(r)$ we have $rs = sr$ for all $r \in R$. Therefore $s \in Z(R)$. Hence the lemma follows.  
\end{proof}

\begin{lemma}\label{lemma2}
If $R$ is a ring, then $R$ is the union of centralizers of all non-central elements of $R$.
\end{lemma}
\begin{proof}
It is clear that $\underset{r  \in R - Z(R)}{\cup} C(r)  \subseteq  R$. Again, for any $s \in Z(R)$, we have by Lemma \ref{lemma1}, $s \in C(r)$ for all $r \in R$. So $s \in \underset{r  \in R - Z(R)}{\cup} C(r)$. Also for any $s \in R-Z(R)$, we have $s \in C(s)$ and so $s \in \underset{r  \in R - Z(R)}{\cup} C(r)$. Hence the lemma follows.  
\end{proof}

\begin{lemma}\label{lemma3}
A ring $R$ cannot be written as a union of two of its proper subrings. 
\end{lemma}
\begin{proof}
The lemma follows from the well-known fact that a group can not be written as a union of two of its proper subgroups.
\end{proof}

\begin{thm}\label{theorem1}
For any non-commutative ring $R$, $|\Cent(R)| \geq 4$.
\end{thm}

\begin{proof}
Since $R$ is non-commutative, so $|\Cent(R)| \geq 2$. If $|\Cent(R)| = 2$, then, by Lemma \ref{lemma2}, $R$ is equal to a proper subset of itself, which is not possible. Also by Lemma \ref{lemma3}, $|\Cent(R)| \neq 3$. Hence the theorem follows. 
\end{proof}
\noindent Note that the ring
$R=\left\lbrace  \begin{bmatrix}
    0 & 0\\
    0 & 0\\
  \end{bmatrix},
 \begin{bmatrix}
    1 & 0\\
    1 & 0\\
  \end{bmatrix},
   \begin{bmatrix}
    0 & 1\\
    0 & 1\\
  \end{bmatrix},
   \begin{bmatrix}
    1 & 1\\
    1 & 1\\
  \end{bmatrix} \right\rbrace$, where $0, 1 \in \Z_2$, has $4$ distinct centralizers. So the above result is the best one possible.

At this point, the following  question, similar to the question posed by Belcastro and Sherman \cite[p. 371]{bG94},  arises naturally.
\begin{Question}
Does there exist an  $n$-centralizer ring for any positive integer $n \ne 2, 3$? Can we characterize an $n$-centralizer ring?
\end{Question}

 The following results show the existence of $n$-centralizer rings for some values of $n$.

\begin{prop}
There exists a $(p + 2)$-centralizer ring for any prime $p$.
\end{prop}

\begin{proof}
We consider the ring $R=\left\lbrace  \begin{bmatrix}
    a & b\\
    0 & 0\\
  \end{bmatrix} \;:\; a, b \in \Z_p \right\rbrace$.
For any element $ \begin{bmatrix}
    x & y\\
    0 & 0\\
  \end{bmatrix}$ of $C\left( \begin{bmatrix}
    a & b\\
    0 & 0\\
  \end{bmatrix}\right)$ we have $xb - ay = 0$.
  
Clearly, $C\left( \begin{bmatrix}
    0 & 0\\
    0 & 0\\
  \end{bmatrix}\right)  = R$. Using simple calculations, we have for any $a \neq 0$ and $l \in \mathbb{Z}_p$,
     
  $C\left( \begin{bmatrix}
    a & 0\\
    0 & 0\\
  \end{bmatrix}\right) = \left\lbrace  \begin{bmatrix}
    x & 0\\
    0 & 0\\
  \end{bmatrix} \;:\; x \in \Z_p \right\rbrace$ and
  
    $C\left( \begin{bmatrix}
    la & a\\
    0 & 0\\
  \end{bmatrix}\right) = \left\lbrace  \begin{bmatrix}
    lx & x\\
    0 & 0\\
  \end{bmatrix} \;:\;  x \in \Z_p \right\rbrace$.  
Hence $|\Cent(R)| = p + 2$.
\end{proof}

The above proposition is a particular case of the following theorem.
\begin{thm}\label{p^2}
Let $R$ be a non-commutative ring of order $p^2$, where $p$ is a prime. Then $|\Cent(R)| = p + 2$.
\end{thm}

\begin{proof}
For any $x \in R - Z(R)$, we consider $C(x)$. As $C(x)$ is an additive subgroup of $R$ we have $|C(x)| = 1, p \;\text{or}\; p^2$. Clearly, $|C(x)| \neq 1, p^2$, as $x, 0_R \in C(x)$ and $R$ is non-commutative, where $0_R$ is the additive identity in $R$. Hence $C(x)$ is additive cyclic group of order $p$ and so $Z(R) = \{0_R\}$. 

Let $x, y \in R - Z(R)$. If  there exists an element  $t (\neq 0_R) \in C(x) \cap C(y)$ then $C(x) = C(y)$, as $C(x), C(y)$ are additive cyclic groups of order $p$. Thus for any  $x, y \in R - Z(R)$ we have either $C(x) \cap C(y) = \{0_R\}$ or $C(x) = C(y)$. Therefore the number of centralizers of non-central elements is $\dfrac{|R| - |Z(R)|}{p - 1} = \dfrac{p^2 - 1}{p - 1} = p + 1$. Hence $|\Cent(R)| = p + 2$.  
\end{proof}

\begin{thm}
Let $p$ be a prime number and $R$ be a non-commutative ring of order $p^3$ with unity. Then $|\Cent(R)| = p + 2$. 
\end{thm}

\begin{proof}
Let $x$ be an arbitrary element of $R - Z(R)$. Then $C(x)$ is an additive subgroup of $R$ and so $|C(x)| = 1, p, p^2 \;\text{or}\; p^3$. Here $|C(x)| \neq 1, p^3$ as $x, 0_R \in C(x)$, where $0_R$ is the additive identity in $R$ and $R$ is non-commutative. If $|C(x)| = p$ then $|Z(R)| = 1$, which is not possible as $0_R, 1_R \in Z(R)$. So $|C(x)| = p^2$ and this gives $|Z(R)| = p$. 

Now, we suppose that $y \in R - Z(R)$ and $y \in C(x)$. Let $z \in C(x)$ be an arbitrary element. We know that $Z(R) \subset Z(C(x))$ and so $|Z(C(x))| > 1$, therefore by Lemma $3$ of \cite{o&v}, $C(x)$ is commutative. Therefore $z \in C(y)$, as $y \in C(x)$. So $C(x) \subseteq C(y)$. Also $|C(x)| = |C(y)|$. Hence, $C(x) = C(y)$; and if $y \notin C(x)$ then $C(x) \cap C(y) = Z(R)$. Therefore the number of centralizers of non-central elements of $R$ is
$\dfrac{|R| - |Z(R)|}{|C(x)| - |Z(R)|} = \dfrac{p^3 - p}{p^2 - p} = p + 1$. Thus $|\Cent(R)| = p + 2$.     
\end{proof}

As an application of the above theorem, it follows that the ring $R=\left\lbrace  \begin{bmatrix}
    a & b\\
    0 & c\\
  \end{bmatrix} \;|\; a, b, c \in \Z_p \right\rbrace$ having order $p^3$ is a  $(p + 2)$-centralizer ring. The following theorem, which is generalization of Theorem \ref{p^2}, gives another class of  $(p + 2)$-centralizer  rings .

\begin{thm}\label{dc}
Let $R$ be a ring and $\frac{R}{Z(R)} \cong \Z_p \times \Z_p$, where $p$ is a prime. Then $|\Cent(R)| = p + 2$.
\end{thm} 

\begin{proof}
We write $Z := Z(R)$. Since  $R/Z \cong \Z_p \times \Z_p$ we have 
\[ 
\frac{R}{Z} = \langle Z + a, Z + b \;:\; p(Z + a) = p(Z + b) = Z; a, b \in R \rangle. 
\] 
If $S/Z$ is additive non-trivial subgroup of $R/Z$ then $|S/Z| = p$. Therefore any additive proper subgroup of $R$ properly containing $Z$ has $p$ \;disjoint right cosets. Hence the proper additive subgroups of $R$ properly containing $Z$ are
\begin{align*}
&S_m = Z \cup (Z + (a + mb)) \cup (Z + 2(a + mb)) \cup \dots \cup (Z + (p - 1)(a + mb)),\\
      & \qquad \text{where } 1 \leq m \leq (p - 1),\\
&S_p = Z \cup (Z + a) \cup (Z + 2a) \cup \dots \cup (Z + (p - 1)a) \text{ and} \\
&S_{p + 1} = Z \cup (Z + b) \cup (Z + 2b) \cup \dots \cup (Z + (p - 1)b).      
\end{align*}
Now for any $x \in R - Z$,  we have $Z + x$ is equal to  $Z + k$ for some $k \in \{ma, mb,   a + mb, 2(a + mb), \dots, (p-1)(a + mb): 1 \leq m \leq (p - 1) \}$. Therefore $C(x) = C(k)$. Again, let $y \in S_j - Z$ for some $j \in \{1, 2, \dots, (p + 1)\}$, then $C(y) \neq S_q$, where $1 \leq q(\neq j) \leq (p + 1)$. Thus $C(y) = S_j$. Hence $|\Cent(R)| = p + 2$.     
\end{proof}

Further, we have the following theorem analogous to  Lemma $2.7$ of \cite{ashrafi00}.

\begin{thm}\label{pring}
Let $R$ be a non-commutative ring whose order is a power of a prime $p$. Then $|\Cent(R)| \geq p + 2$, and equality holds if and only if $\frac{R}{Z(R)} \cong \Z_p \times \Z_p$.
\end{thm}
\begin{proof}
Let $R$ be a non-commutative ring whose order is a power of a prime $p$. Suppose $k = |\Cent(R)|$.
 Let $A_1, \dots, A_k$ be the distinct centralizers of $R$ such that $|A_1| \geq \cdots \geq |A_k|$ and $A_1 = R$. So $R = \underset{i = 2}{\overset{k}{\cup}} A_i$ and by Cohn's theorem in \cite{Cohn}, we have $|R| \leq \underset{i = 3}{\overset{k}{\sum}} |A_i|\;\;(\text{as}\; A_i\text{'s are additive groups})$.
Also $|A_i| \leq \frac{|R|}{p}$, where $i \neq 1$. Hence
\[
|R|  \leq \underset{(k - 2)- \text{times}}{\underbrace{\frac{|R|}{p} + \cdots + \frac{|R|}{p}}}
\]
which implies $|R|  \leq (k -2)\frac{|R|}{p}$ and so $k \geq p + 2$. That is $|\Cent(R)| \geq p + 2$.

For the equality, if $\frac{R}{Z(R)} \cong  \Z_p \times \Z_p$ then by Theorem \ref{dc}, we have $|\Cent(R)| = p + 2$. Conversely, we assume that $l = |\Cent(R)| = p + 2$. Suppose $A_1, A_2, \dots, A_l$ are distinct centralizers of $R$ such that $|A_1| \geq \cdots \geq |A_l|$ and $A_1 = R$. So $R = \underset{i = 2}{\overset{l}{\cup}} A_i$ and by Cohn's theorem in \cite{Cohn}, we have $|R| \leq \underset{i = 3}{\overset{l}{\sum}} |A_i|.$
Also $|A_i| \leq \frac{|R|}{p}$, where $i \neq 1$. Suppose, there exists an $A_i$ such that $|A_i| < \frac{|R|}{p}$ for $3 \leq i \leq l$ then
\[
|R|  < \underset{(l - 2)- \text{times}}{\underbrace{\frac{|R|}{p} + \cdots + \frac{|R|}{p}}} = (l - 2)\frac{|R|}{p} = |R|,
\]
 a contradiction. Hence $|A_3| = \frac{|R|}{p}, \dots, |A_l| = \frac{|R|}{p}$. Also $|A_2| \geq \dots  \geq |A_l|$, so $|A_i| = \frac{|R|}{p}$, where $2 \leq i \leq l$. Hence $ \overset{l}{\underset{i = 3}{\sum}}|A_i| = (l - 2)\frac{|R|}{p} = |R|$. Therefore $ \overset{l}{\underset{i = 3}{\sum}}|A_i| = |R|$ if and only if $A_2 + A_m = R$, for all $m \neq 2$ and $A_k \cap A_l \subseteq A_2$ for all $k \neq l$ (By Cohn's Theorem in \cite{Cohn}). Interchanging $A_i$'s we have $A_2 \cap A_3 = Z(R)$. Thus 
\[|R| = |A_2 + A_3| = \frac{|A_2||A_3|}{|A_2 \cap A_3|}  = \frac{|R|^2}{p^2|Z(R)|} 
\]
which gives $| R:Z(R)| = p^2 $. Hence $\frac{R}{Z(R)} \cong  \Z_p \times \Z_p$,
since $R$ is non-commutative. This completes the proof. 
\end{proof}

We conclude this section by the following result.
\begin{prop}
There exists an  $8$-centralizer ring.  
\end{prop}
\begin{proof}
We consider the ring $R = \{a + bi + cj + dk : a, b, c, d \in \Z_p, i^2 = j^2 = k^2 = -1, ij=k, jk=i, ki=j, ji=-k, kj=-i, ik=-j\}$.  If $b = c = d = 0$, then clearly $C(a) = R$. If $c = d = 0$ and  $b \neq 0$, then $C(a + bi)= \{x + yi : x, y \in \Z_p\}$. If $b = d =0$ and  $c \neq 0$, then $C(a + cj)= \{x + zj : x, z \in \Z_p\}$. If $b = c =0$ and  $d \neq 0$, then $C(a + dk)= \{x + wk : x, w \in \Z_p\}$. If $d =0$ and  $b, c \neq 0$, then $C(a + bi +cj)= \{x + yi + zj : bz = cy, x, y, z \in \Z_p\}$. If $c = 0$ and  $b, d \neq 0$, then $C(a + bi + dk)= \{x + yi + wk : bw = dy, x, y, w \in \Z_p\}$. If $b =0$ and  $c, d \neq 0$, then $C(a + cj + dk)= \{x + zj + wk : cw = dz, x, z, w \in \Z_p\}$. If $b, c, d \neq 0$, then $C(a + bi + cj + dk)= \{x + yi + zj + wk : bz= cy, dy = bw, dz= cw, x, y, z, w \in \Z_p\}$. Hence $|\Cent(R)| = 8$. 
\end{proof}

\section{$4$-centralizer rings}
In this section, we give a characterization of finite $4$-centralizer rings analogous to  Theorem 2 of \cite{bG94}. The following lemma which is useful in characterization of $4$-centralizer rings. 

\begin{lemma}\label{lemma4}
Let $R$ be a $4$-centralizer finite ring. Then at least one of the centralizers of non-central elements has index $2$ in $R$.
\end{lemma}

\begin{proof}
Let $A, B, C$ be the three proper centralizers of $R$. Suppose none of $A, B, C$ has index $2$, that is $|R : A| \geq 3, |R : B| \geq 3, |R : C| \geq 3$. Then
as $R = A \cup B \cup C $, we have
\[
|R|  \leq  |A| + |B| + |C| - 2|Z(R)|  \leq \frac{|R|}{3} + \frac{|R|}{3} + \frac{|R|}{3} - 2|Z(R)| < |R|,
\]
%\begin{align*}
%R &= A \cup B \cup C \\
%\Rightarrow |R| & \leq  |A| + |B| + |C| - 2|Z(R)| \\
%& \leq \frac{|R|}{3} + \frac{|R|}{3} + \frac{|R|}{3} - 2|Z(R)| < |R|,
%\end{align*} 
which is a contradiction. Hence the lemma follows.
\end{proof}

%Now we prove the following theorem which characterizes $4$-centralizer rings. Here the technique of proving the converse part of the following theorem is taken from \cite[Theorem 2]{bG94}.

We have the following characterization of finite $4$-centralizer rings.
 
\begin{thm}\label{4c}
Let $R$ be a non-commutative finite ring. Then $|\Cent(R)| = 4$ if and only if $\frac{R}{Z(R)} \cong \Z_2 \times \Z_2$.
\end{thm}

\begin{proof}
If $\frac{R}{Z(R)} \cong \Z_2 \times \Z_2$ then by Theorem \ref{dc}, we have    $|\Cent(R)|  =  4$. 

Conversely, let $|\Cent(R)| = 4$ then $R$ has exactly four distinct centralizers, say $R, A, B, C$ where $A, B, C$ are centralizers of three distinct non-central elements of $R$.  

By Lemma \ref{lemma3}, $R$ cannot be written as the union of two of its proper subrings of $R$. Therefore we may choose $a \in A - (B \cup C), b \in B - (C \cup A), c \in C - (A \cup B)$ respectively. It can be easily seen that $C(a) = A, C(b) = B, C(c) = C$. By Lemma \ref{lemma4}, at least one of the centralizers $A, B, C$, say $A$ has index $2$ in $R$, that is $|R : A| = 2$.

Now, let $x \in (A \cap B) - Z(R)$ then $C(x) \neq R$. If $C(x) = A$ then $a, b \in C(x)$. So, $C(x) \neq A$. Similarly it can be seen that $C(x) \neq B$. If $C(x) = C$ then $x \in A \cap B \cap C = Z(R)$ (using Lemma \ref{lemma1}), which is a contradiction. Therefore $|\Cent(R)|$ must be at least $5$, which is again a contradiction. So $A \cap B = A \cap B \cap C = Z(R)$. Similarly it can be seen that  $B \cap C = Z(R), A \cap C = Z(R)$. Again $A, B, C$ are additive subgroups of $R$, therefore
\[
|R| \geq |A + B| = \frac{|A||B|}{|A \cap B|} = \frac{|A||B|}{|Z(R)|}
\]
which gives $|B| \leq 2|Z(R)|$.
%\begin{align*}
%&|A + B| = \frac{|A||B|}{|A \cap B|} \\
%\Rightarrow &|A \cap B| \geq \frac{|A||B|}{|R|} \\
%\Rightarrow &|B| \leq 2|Z(R)|.
%\end{align*}   
Since $Z(R) \subset B$, so $\frac{|B|}{2} \leq |Z(R)| < |B|$. Hence $|B| = 2|Z(R)|$. Similarly $|C| = 2|Z(R)|$. Therefore
\[
|R| = |A| + |B| + |C| - 2|Z(R)|  = \frac{|R|}{2}  +  2|Z(R)|
\]
which gives $|R : Z(R)| = 4$ and hence $\frac{R}{Z(R)} \cong \Z_2 \times \Z_2$.
%\begin{align*}
%&|R| = |A| + |B| + |C| - 2|Z(R)| \\
%\Rightarrow &\frac{|R|}{2}  = 2|Z(R)| \\
%\Rightarrow &|R : Z(R)| = 4\\
%\Rightarrow &\frac{R}{Z(R)} \cong \Z_2 \times \Z_2.
%\end{align*} 
%This completes the proof of the theorem.
\end{proof}

\section{$5$-centralizer rings}
In this section, we give a characterization of finite $5$-centralizer rings analogous to  Theorem $4$ of \cite{bG94}. The following lemmas are   useful in this regard.

%In this section, we study the structure of $5$-centralizer rings. Here also the technique of proving the main theorem  is same as Theorem $4$ in \cite{bG94}. 
 
\begin{lemma}\label{lemma5C1}
Let $R$ be a ring and $R = A \cup B \cup C$, where $A, B, C$ are the proper distinct subrings. We put $K  = A \cap B \cap C, L = A \cap B - K, M = A \cap C - K$, $N = B \cap C - K$ and $A' = A - (B \cup C)$, $B' = B - (A \cup C)$, $C' = C - (A \cup B)$. Then 
\bnum
\item $L = M = N = \phi$,
\item  $A' + B' \subseteq C', B' + C' \subseteq A'$ and $C' + A' \subseteq B'$,
\item  $A' + A' \subseteq K, B' + B' \subseteq K$ and $C' + C' \subseteq K$,
\item  $|R : K| = 4$.
\enum 
\end{lemma}

\begin{proof}
(a)  We consider $l \in L$ and $c' \in C'$. Then $c' + l \in A$ or $B$ or $C$. If $c' + l \in A$ then $c' + l + (-l) = c' \in A$, a contradiction. If $c' + l \in B$ then $c' + l + (-l) = c' \in B$, a contradiction. If $c' + l \in C$ then $(-c') + c' + l = l \in C$, a contradiction. Since $C' \neq \phi$, we must have $L = \phi$. Similarly $M = N = \phi$. 

(b) Let $a' \in A'$, then $a' \in A \Rightarrow -a' \in A \Rightarrow -a' \in K$ or $A'$. If $-a' \in K$ then $a' \in K$, a contradiction. Hence  $-a' \in A'$. Similarly if $b' \in B'$ then $-b' \in B'$ and if  $c' \in C'$ then $-c' \in C'$. Suppose $a' \in A', b' \in B'$ then $a' +  b' \in K$ or $A'$ or $B'$ or $C'$. If $a' + b' \in A' \subseteq A$ then $b'= -a' + a' + b' \in A$, a contradiction. If $a' + b' \in B' \subseteq B$, then $a' = a' + b' + (-b') \in B $, a contradiction. If $a' + b' \in K$, then $a' + b' \in A$, a contradiction. Hence $a' + b' \in C'$. Thus $A' + B' \subseteq C'$. Similarly it can be seen that $B' + C' \subseteq A'$ and $C' + A' \subseteq B'$.

(c) Let $a', {a_1}' \in A' \subseteq A$. So $a' + {a_1}' \in A \Rightarrow \; a' + {a_1}' \in A'$ or $K$. Let $a' + {a_1}' \in A'$. We consider $b' + a' + {a_1}'$, for some $b' \in B'$. Then by second part we have $b' + (a' + {a_1}') \in C'$ and $(b' + a') + {a_1}' \in B'$. So $b' + a' + {a_1}' \in B' \cap C'$, a contradiction. Similarly we can show the other two.

(d) From part (a), we have $R = K \cup A' \cup B' \cup C'$. Let $k + a' \in K + a'$ where $k \in K, a' \in A'$ then $k + a' \in A = K \cup A'$. If $k + a' \in K$ then $a' \in K$, a contradiction. So $K + a' \subseteq A'$. Again $x' \in A'$ gives $x' + (-a') \in K $ (by part (c)). So, $x' \in K + a'$. Hence $K + a' = A'$. Similarly it can be seen that $K + b' = B', K + c' = C'$, where $b' \in B', c' \in C'$. Therefore $|R : K| = 4$.  
\end{proof}

\begin{lemma}\label{lemma5C2}
Let $R$ be a $5$-centralizer finite ring and  $A, B, C, D$ be the four proper centralizers of $R$.
% and we choose $a \in A - (B \cup C), b \in B - (A \cup C)$ and $c \in C - (A \cup B)$. 
Then
\bnum
%\item No one of $A, B, C$ or $D$ is contained in the union of the other three.
%\item No element of $R$ is in exactly two proper centralizers.
%\item No element of $R$ is in exactly three proper centralizers.
\item $|R| = |A| + |B| + |C| + |D| - 3|Z(R)|$.
\item If $S$ and $T$ are distinct proper centralizers of $R$, then 
\[
\frac{|S||T|}{|R|} \leq |Z(R)| \leq \frac{|R|}{6}. 
\]
\enum
\end{lemma}

\begin{proof}
Let $a \in A - (B \cup C), b \in B - (A \cup C)$ and $c \in C - (A \cup B)$. Suppose there does not exist any $a \in A - (B \cup C)$ such that $C(a) = A$. Then $C(a) = D$ for all $a \in A - (B \cup C)$. Therefore $A - (B \cup C) \subseteq D - (B \cup C)$. Interchanging the roles of $A$ and $D$ we get $A - (B \cup C) = D - (B \cup C)$, which gives $A \cup B \cup C = D \cup B \cup C = R$. Again, by Lemma \ref{lemma5C1} (a), we have $B \cap C = C \cap D$ and so $Z(R) = A \cap B \cap C$. Therefore, by Lemma \ref{lemma5C1} (d), we have $R/Z(R) \cong \Z_2 \times \Z_2$. This gives $|\Cent(R)| = 4$, contradiction. Hence  $C(a) = A$. Similarly $C(b) = B$ and $C(c) = C$.

%If $C(a) \neq A$ then $C(a) = B, C$ or $D$. But $a \notin B, C$; so $C(a) \neq B, C$. So $a \in D - (B \cup C)$ and this gives $A - (B \cup C) \subseteq D - (B \cup C) \Rightarrow A \subseteq D$. Next if we interchange the roles of $A$ and $D$, then we have $D \subseteq A$. Hence $A = D$, a contradiction. Therefore $C(a) = A$. Similarly $C(b) = B, C(c) = C$.

(a)  Let us assume without loss of generality that $D$ is a subset of $A \cup B \cup C$. Then $R = A \cup B \cup C \cup D = A \cup B \cup C$. Now, by Lemma \ref{lemma5C1}, we have $|R : K| = 4$ where $K = A \cap B \cap C = Z(R)$. 
%We shall show that $A \cap B \cap C \; = \; Z(R)$. Let $ x \in (A \cap B \cap C) - Z(R)$, then $C(x) \neq R$. Also $C(x) \neq D$, since $C(x) = D$ gives $x \in Z(R)$, contradiction. Moreover $C(x) \neq A, B, C$, as $x \in A \cap B \cap C$ gives $a, b, c \in C(x)$, but $b \notin A, a \notin B$ and $a \notin C$. This means if $ x \in A \cap B \cap C - Z(R)$, then $C(x) \neq R, A, B, C, D$. That is $| \Cent(R)|$ must be at least $6$, a contradiction. Hence $A \cap B \cap C = Z(R)$. 
Thus by Theorem \ref{4c}, $|\Cent(R)| = 4$,                                   which is a contradiction. Therefore no one of $A, B, C$ or $D$ is contained in the union of the other three. 

 Let $r \in (A \cap B) - (C \cup D)$ then $r \in C(a) \cap C(b)$ which gives $a, b \in C(r)$. But $a \notin C(b)$, so $C(r) \neq A, B$. Again $r \notin C, D$; so $C(r) \neq C, D$. Also $C(r) \neq R$, since $r \in R - Z(R)$. Therefore $| \Cent(R)|$ must be at least $6$, a contradiction. Hence $(A \cap B) - (C \cup D) = \phi$. This shows that no element of $R$ is in exactly two proper centralizers.

 Let $r \in (A \cap B \cap C) - D$ then $r \in C(a) \cap C(b) \cap C(c)$. Therefore $a, b, c \in C(r)$. But $b \notin C(a), c \notin C(b)$. So $C(r) \neq A, B, C$. Also $C(r) \neq D, R$; as $r \notin D$ and $r \notin Z(R)$. Therefore $| \Cent(R)|$ must be at least $6$, a contradiction. Hence $A \cap B \cap C - D = \phi$. Thus no element of $R$ is in exactly three proper centralizers.

 From above, it can be seen clearly that
\[
|R| = |A \cup B \cup C \cup D| = |A| + |B| + |C| + |D| - 3|Z(R)|. 
\]

(b) Note that for any two proper centralizers $S$ and $T$ of $R$ we have $S \cap T = Z(R)$, since no element of $R$ is in exactly two as well as three proper centralizers. 
%Given $S$ and $T$ are proper centralizers of $R$, so $S$ and $T$ are two additive subgroups of $R$ and using above parts we have $S \cap T = Z(R)$. 
Also any  proper centralizers of $R$  are additive subgroups of $R$, so
 $\frac{|S||T|}{|S + T|} = |S \cap T| = |Z(R)|$. Since $S + T \subseteq R$ we have  $|Z(R)|  \geq \frac{|S||T|}{|R|}$.

Again by part (a), 
\begin{align*}
|R| & \; = \; |A| + |B| + |C| + |D| - 3|Z(R)| \\
& \geq 2|Z(R)| + 2|Z(R)| + 2|Z(R)| + 2|Z(R)| - 3|Z(R)|. 
\end{align*}
Thus $|R : Z(R)| \geq 5$. If $|R : Z(R)| = 5$ then $\frac{R}{Z(R)} \cong {\mathbb{Z}}_5$, a contradiction. Therefore $|Z(R)| \leq \frac{|R|}{6}$. So, $\frac{|S||T|}{|R|} \leq |Z(R)| \leq \frac{|R|}{6}$.
\end{proof}
We would like to mention here that the group theoretic analogues of Lemma \ref{lemma5C1} and Lemma \ref{lemma5C2} can be found in \cite{BBM70} and \cite{bG94} respectively. Now we   prove the main theorem of this section which characterizes finite $5$-centralizer rings.

\begin{thm}
Let $R$ be a finite ring. Then $|\Cent(R)| = 5$ if and only if  $\frac{R}{Z(R)} \cong  \Z_3 \times \Z_3$. 
\end{thm}
\begin{proof}
Let $\frac{R}{Z(R)} \cong  \Z_3 \times \Z_3$, then by Theorem \ref{dc}, we get $|\Cent(R)| = 5$.

Conversely, let $| \Cent(R)| = 5$. Let $A, B, C, D$ be the four proper centralizers of $R$. Then by Lemma \ref{lemma5C2} (b), $\frac{|A||B|}{|R|} \leq |Z(R)| \leq \frac{|R|}{6}$. Our aim is to get more near lower bound for $|Z(R)|$. We may assume without loss of generality that $|A| \geq |B| \geq |C| \geq |D|$. Suppose $|A| < \frac{|R|}{3}$, as $1 < |A| \leq \frac{|R|}{2}$. That is $|A| \leq \frac{|R|}{4}$. Now by Lemma \ref{lemma5C2} (a), $|R| \leq  |R| - 3|Z(R)|  < |R|$, a contradiction. Hence $|A| = \frac{|R|}{2}$ or $|A| = \frac{|R|}{3}$. If $|A| = \frac{|R|}{2}$, then $|R| = |A| + |B| + |C| + |D| - 3|Z(R)|$ gives $\frac{|R|}{2}  < |B| + |C| + |D|$ and so $\frac{|R|}{6}  < |B|$.
Also, applying Lemma \ref{lemma5C2} (b) on $A$ and $B$ we have $\frac{|R|}{6} < |B| \leq \frac{|R|}{3}$. So $|B|$ is one of $\frac{|R|}{3}, \frac{|R|}{4}$ or $\frac{|R|}{5}$. Reapplying Lemma \ref{lemma5C2} (b)  on $A$ and $B$ we have,
\[
\frac{|A||B|}{|R|}  \leq |Z(R)| \leq \frac{|R|}{6}
\] 
which gives $\frac{|R|}{10} \leq |Z(R)| \leq \frac{|R|}{6}$. 
Thus $|Z(R)|$ is one of $\frac{|R|}{6}, \frac{|R|}{7}, \frac{|R|}{8}, \frac{|R|}{9}$ or $\frac{|R|}{10}$. Let $|Z(R)| = \frac{|R|}{7}, \frac{|R|}{9}$ then $2$ divides $7$ and $9$, which is not possible. If $|Z(R)| = \frac{|R|}{6}$ then  $\frac{R}{Z(R)} \cong \Z_6$, a contradiction. Let $|Z(R)| = \frac{|R|}{8}$ then $\frac{|R|}{8}$ divides $|B|$. If $|B| = \frac{|R|}{3}, \frac{|R|}{5}$ then $3, 5$ divides $8$, a contradiction. Therefore $|B| = \frac{|R|}{4}$. By Lemma \ref{lemma5C2} (a), we have $\frac{5|R|}{8} = |C| + |D|$. Also $|B| \geq |C| \geq |D|$. So $|C| + |D| \leq \frac{|R|}{2} < \frac{5|R|}{8} = |C| + |D|$, a contradiction. If $|Z(R)| = \frac{|R|}{10}$, then $\frac{|R|}{10}$ divides $|B|$. If $|B| = \frac{|R|}{3}, \frac{|R|}{4}$ then $3, 4$ divides $10$, a contradiction. Therefore $|B| = \frac{|R|}{5}$. Now Lemma \ref{lemma5C2}(a) gives, $|C| + |D| = \frac{6|R|}{10}$. Also $|B| \geq |C| \geq |D|$, therefore $|C| + |D| \leq \frac{2|R|}{5} < \frac{6|R|}{10} = |C| + |D|$, a contradiction. 

If $|A| = \frac{|R|}{3}$ then  Lemma \ref{lemma5C2} (a) gives, $\frac{2|R|}{3} < |B| + |C| + |D|$ which gives $\frac{2|R|}{3}  < 3|B|$ and so $|B| \geq \frac{|R|}{4}$.
Also $|A| \geq |B|$, so $|B| = \frac{|R|}{3}$ or $\frac{|R|}{4}$. Again, applying Lemma \ref{lemma5C2} (b) on $A$ and $B$ we get,
\[
\frac{|A||B|}{|R|} \leq |Z(R)|  \leq \frac{|R|}{6} 
\]
which gives
$\frac{|R|}{12} \leq |Z(R)| \leq \frac{|R|}{6}$. Therefore $|Z(R)|$ is one of $\frac{|R|}{6}, \frac{|R|}{7}, \frac{|R|}{8}, \frac{|R|}{9}, \frac{|R|}{10}$, $\frac{|R|}{11}$ or $\frac{|R|}{12}$. Now if $|Z(R)| \; = \; \frac{|R|}{7}, \frac{|R|}{8}, \frac{|R|}{10}, \frac{|R|}{11}$ then $3$ divides $7, 8, 10, 11$, a contradiction. Let $|Z(R)| = \frac{|R|}{6}$ then as above we get a contradiction. Let $|Z(R)| = \frac{|R|}{9}$ then $\frac{R}{Z(R)} \cong \Z_3 \times \Z_3$. Let $|Z(R)| = \frac{|R|}{12}$ and $|B| = \frac{|R|}{3}$ then applying Lemma \ref{lemma5C2} (b) on $A$ and $B$ we have, $\frac{|R|}{9} \leq \frac{|R|}{12}$, a contradiction. If $|B| = \frac{|R|}{4}$ then Lemma \ref{lemma5C2} (a) gives, $|C| + |D| = \frac{4|R|}{6}$.
 Also $|C|, |D| \leq \frac{|R|}{4}$, so $|C| + |D| \leq \frac{3|R|}{6} < \frac{4|R|}{6} = |C| + |D|$, which is not possible. Hence $\frac{R}{Z(R)} \cong \Z_3 \times \Z_3$.
\end{proof}

%After proving the above theorems one question naturally arises, ``if $R$ is a non - commutative ring then, is it true that $|\Cent(R)| = p + 2$ if and only if $\frac{R}{Z(R)} \cong \Z_p \times \Z_p$"

%The following theorem gives the answer of this question partially. 

\section{Relation between $|\Cent(R)|$ and $d(R)$}
Note that $d(R) = 1$ if and only if $R$ is commutative. Therefore, by Proposition \ref{prop1}, we have  $|\Cent(R)| = 1$ if and only if $d(R) = 1$. By Theorem \ref{4c} and Theorem $1$ of \cite{dmachale}, we have
the following result.

\begin{prop}
Let $R$ be a non-commutative finite ring. Then $|\Cent(R)| = 4$ if and only if $d(R) = \frac{5}{8}$.
\end{prop}
In \cite{dmachale}, MacHale also proved the following theorem:
\begin{thm}\label{dm}
Let $R$ be a non-commutative finite ring and $p$ the smallest prime dividing the  order of $R$. Then $d(R) \leq \frac{1}{p^3}(p^2 + p - 1),$ with equality  if and only if $|R : Z(R)| = p^2$. 
\end{thm}
\noindent Now by Theorem \ref{dc} and Theorem \ref{dm}, we have the following interesting connection between $d(R)$ and $|\Cent(R)|$.
\begin{prop}\label{rc}
Let $R$ be a non-commutative finite ring and $p$ the smallest prime dividing the  order of $R$. If $d(R) = \frac{1}{p^3}(p^2 + p - 1)$ then  $|\Cent(R)| = p + 2$. 
\end{prop}
%We conclude the paper by noting that the converse of the above proposition is also true for any non-commutative ring whose order is a power of a prime $p$.
We conclude the paper by noting that the converse of  Proposition \ref{rc} holds for some finite non-commutative rings. In particular, by Theorem \ref{pring} and Theorem \ref{dm}, we have the following result.
\begin{prop}
 Let $R$ be a non-commutative ring whose order is a power of a prime $p$. If  $|\Cent(R)| = p + 2$ then $d(R) = \frac{1}{p^3}(p^2 + p - 1)$. 
\end{prop}
%Finally we pose the following problem:

%If $R$ is a non-commutative ring with order divisible by the prime $p$ but by no smaller prime then, is it true that $ |\Cent(R)| = p + 2 \;\text{implies}\; d(R) = \frac{1}{p^3}(p^2 + p - 1)?$ 

%\noindent \textbf{Acknowledgement:} The author would like to thank the referee for his/her valuable suggestions to modify the proof of the main theorem.

\end{document}